\documentclass[preprint,12pt]{elsarticle}
\usepackage{amssymb}
\usepackage{amsmath}
\usepackage{amssymb}
\usepackage{amsmath}
\usepackage{enumitem}
\usepackage[utf8]{inputenc}
\usepackage[numbers]{natbib}
\usepackage{breqn}
\usepackage[title]{appendix}
\usepackage{quotmark}
\usepackage{amsthm} 
\usepackage{varioref}
\usepackage{graphicx}
\newtheorem{theorem}{Theorem}[section]
\newtheorem{lemma}[theorem]{Lemma}
\newtheorem*{remark*}{Remark}
\newtheorem{corollary}[theorem]{Corollary}

\theoremstyle{definition} 
\newtheorem{definition}{Definition}[section]
\newtheorem{example}{Example}[section]
\newtheorem{remark}{Remark}[section]

\usepackage{hyperref}
\usepackage{cleveref}
\journal{Topology and its applications}

\begin{document}

\begin{frontmatter}
\title{Open cover specification property and chaos}

\author[aff1]{Shweta Wadhwani}
\ead{wadhwanishweta175@gmail.com}
\author[aff1]{Devender Kumar}
\ead{kalgandevender@gmail.com}
\author[aff2]{A R Prasannan }
\ead{arprasannan@mac.du.ac.in}

  \affiliation[aff1]{organization={Department of Mathematics, University of Delhi},
   postcode={110007},
   state={Delhi},
         country={India} 
}
\affiliation[aff2]{organization={Department of Mathematics, Maharaja Agrasen College, University of Delhi},
           postcode={110007},
           state={Delhi},
         country={India} }

\begin{abstract}
This paper investigates the open cover specification property and its role in the study of topological dynamical systems. We show that this property is a topological invariant and a natural extension of the classical topological specification property from uniform spaces to general topological spaces, and that it is preserved under infinite products with compositions. We further prove that on regular spaces, the strong open cover specification property implies Devaney chaos, and that on locally compact, first countable Hausdorff spaces, it guarantees both distributional chaos and Li–Yorke chaos. These results demonstrate the significance of the open cover specification property in linking topological structure with dynamical complexity.
\end{abstract} 

\end{frontmatter}
\section{Introduction}

Let $f$ be a continuous self-map on $X$, where $X$ is an arbitrary, not necessarily compact, topological space. We denote the set of integers by $\mathbb{Z}$ and the set of natural numbers by $\mathbb{N}$. Specification is a fundamental orbit-approximation property in topological dynamics, which allows finitely many orbit segments to be approximately traced by the orbit of a single point. More precisely, for any $\epsilon>0$, there exists $N\in\mathbb{N}$ such that for any $k\in\mathbb{N}$, any points $x_1,x_2,\ldots,x_k\in X$, and any integers $a_1\leq b_1<a_2\leq b_2<\cdots<a_k\leq b_k,
\qquad a_{i+1}-b_i>N,$
there exists $x\in X$ such that $d\bigl(f^j(x),f^j(x_i)\bigr)<\epsilon$
for $a_i\leq j\leq b_i$ and $1\leq i\leq k$. Thus, the specification property provides a mechanism for **joining** finitely many prescribed orbit segments, with sufficiently long gaps between consecutive segments. This property plays a central role in understanding the long-term behavior of dynamical systems, particularly in the study of periodic points, invariant measures, entropy, and chaotic phenomena. Because specification provides a powerful mechanism for combining orbit segments, it has become an important tool for linking structural properties of dynamical systems with various notions of chaos.

The specification property was introduced in the context of Axiom A diffeomorphisms by Bowen~\cite{bowen1971}, where it was used to study periodic points and invariant measures. Since then, several related forms, such as weak specification, periodic specification, and almost specification, have been developed and studied~\cite{dateyama1990,kulczycki2014,pfister2005}. Pfister and Sullivan~\cite{pfister2007} further generalized specification through approximate and almost product properties, showing that these notions are useful in establishing large deviation principles and entropy formulas. Beyond compact metric spaces, Shah et al.~\cite{shah2016} introduced topological specification for non-compact uniform topological spaces, and more recently, Kumar and Das~\cite{kumar2024} proposed the open cover specification property, extending specification to arbitrary topological spaces.

The study of chaos has also revealed strong connections with specification-type properties. Sensitivity to initial conditions is a fundamental property of dynamical systems in which arbitrarily small differences in initial points can lead to significant differences in their future trajectories. More precisely, a dynamical system $(X,f)$ is said to exhibit sensitive dependence on initial conditions if there exists $\epsilon>0$ such that, for every $x\in X$ and every neighborhood $G$ of $x$, there exist $y\in G$ and $n\in\mathbb{N}$ satisfying $d\bigl(f^n(y),f^n(x)\bigr)>\epsilon.$
Devaney~\cite{devaney1989} characterized chaos through topological transitivity and the density of periodic points. Under the usual assumptions on the underlying space, these properties imply sensitivity to initial conditions, establishing an important connection between topological and metric aspects of chaotic dynamics.

Chaos was first formalized by Li and Yorke~\cite{li1975}, who showed that the existence of a period-three point implies complex dynamical behavior. Schweizer and Smital~\cite{schweizer1994} introduced distributional chaos, providing a refined description of irregular behavior through distributionally scrambled pairs and generalizing Li--Yorke chaos. The relationship between specification and distributional chaos was established by Sklar and Smital~\cite{sklar2000}, who proved that, on compact metric spaces, weak specification implies distributional chaos. These ideas were further refined by Oprocha~\cite{oprocha2010}, who classified distributional chaos into types DC1, DC2, and DC3, and by Yadav and Shah~\cite{yadav2022}, who extended the implications of weak specification to distributional chaos in non-compact spaces.

In this paper, we study the open cover specification property and its consequences for both the structural and chaotic behavior of dynamical systems. We first prove that if the countable product of dynamical systems satisfies the open cover specification property then each factor system also satisfies it and open cover specification property on $(X,f)$ can be extended to open cover specification property on $(X,f^p)$ for any $p \in \mathbb{Z}$. We then show that open cover specification property is a topological conjugacy invariant property. For uniform spaces, we establish the equivalence between the open cover specification property and the topological specification property. Moreover, for compact dynamical systems, we show that chain mixing together with the finite shadowing property implies the open cover specification property. Finally, we investigate chaotic dynamics and provide conditions under which the open cover specification property guarantees Devaney chaos distributional chaos and Li-Yorke chaos.

\section{Preliminaries}
 Unless mentioned otherwise,  $(X,f)$ is a topological dynamical system on . For integers $a,b$ define $[a,b]=\{x\in \mathbb{Z}, a \leqslant x \leqslant b\}$. For two dynamical systems $(X_1,f_1)$ and $(X_2,f_2)$, $(X_1 \times X_2,f_1 \times f_2)$ is the cross product of dynamical systems $X_1$ and $X_2$, with $(f_1 \times f_2)(x_1,x_2)=(f_1(x_1),f_2(x_2))$. For dynamical systems $(X_1,f_1)$ and$(X_2,f_2)$, a property P is said to be a topological conjugacy invariant property if there exists homeomorphism $h: X_{1} \longrightarrow X_{2}$  such that $h \circ f_{1}=f_{2} \circ h$. 

 For a topological dynamical system $(X,f)$, $x \in X$ is periodic point if for some natural number $p$, $f^p(x)=x$. 
 For an open cover $\mathcal{U}$, $(x_i)_{i \in \mathbb{N}}$ is called $\mathcal{U}$-pseudo orbit if $( f(x_{i}),x_{i+1}) \in U \times U$ for some $U \in \mathcal{U}$ for all $i \in \mathbb{N}$.
 For a uniform space $X$ and entourage $\mathbf{V} \in \mathbb{U}$ where $\mathbb{U}$ is a uniformity on X, $(x_i)_{i \in \mathbb{N}}$ is called $\mathbf{V}$-pseudo orbit if for all $i \in \mathbb{N}$, $( f(x_{i}),x_{i+1}) \in \mathbf{V}$. 
 For any finite or infinite sequence $(x_i)_{i \in K}$, where $K \subseteq \mathbb{N}$ and any open cover $\mathcal{U}$ of $X$, $(x_i)$ is said to be $\mathcal{U}$-traced by $y \in X$ if $\left(f^{i}(y), x_{i}\right) \in U \times U, U \in  \mathcal{U}$ and $i \in K$. For any finite or infinite sequence $(x_i)_{i \in K}$ in a uniform space $(X,f)$ where $K \subseteq \mathbb{N}$ and any symmetric neighbourhood $\mathbf{V}$ of $X$, $(x_i)$ is said to be $\mathbf{V}$-traced by $y \in X$ if $\left(f^{i}(y), x_{i}\right) \in  \mathbf{V}$ for every $i \in K$.
  and is written as $\xi=\left\{f^{\left[a_{i}, b_{i}\right]}\left(x_{i}\right)\right\}_{i=1}^{n}$.

 We begin by recalling several basic notions from topological dynamics that will be used throughout the paper.

 \begin{definition}
     For $x \in X$ and an open cover $\mathcal{U}$, $\mathcal{U}$ is said to have an open refinement that is locally finite at the point $x$ if there exists a neighborhood of $x$ meeting only finitely many sets of the refinement.
 \end{definition}

    \begin{definition} 
    For $u,v \in X$, $(u,v)$ is called a pair of distal points if there exists an open cover $\mathcal{U}$ of $X$ such that if $f^{i} (u) \in U_{1} \in \mathcal{U},f^{i} (v) \in U_{2} \in \mathcal{U}$ then $U_{1} \cap U_{2}=\phi, \forall \  i \in \mathbb{N} \cup\{0\}.$
\end{definition}

\begin{definition}
    $(X,f)$ is said to be topological transitivity if for any open subsets $U,V$ of $X$, there exists $n \in \mathbb{N} \cup \{0\}$ such that $U \cap f^n(V) \neq \phi$ and $(X,f)$ is said to be topological mixing if $U \cap f^m(V) \neq \phi, \forall m \geqslant n$.
\end{definition}

\begin{definition}
\cite{Fedeli2018}For a topological space $X$, the dynamical system $(X,f)$ is said to have sensitive dependence on initial conditions(SDIC) if there exists an open cover $\mathcal{U}$ such that for any $x \in X$ and any open neighborhood $G$ of $x$, there exists $y \in G$ and $n \in \mathbb{N}$ such that $\forall U \in \mathcal{U}, (f^n(y),f^n(x)) \notin U\times U $.
Such open cover is called sensitivity cover. 
\end{definition}

\begin{definition}
    A dynamical system $(X,f)$ is said to have chain mixing if for any open cover $\mathcal{U}$ of $X$ and any $u, v \in X$, there exists $N \in \mathbb{N}$ such that for any $n \geqslant N $ there exists a finite $\mathcal{U}$-pseudo orbit $(x_i)_{i=0}^n$ where $u=x_{0}$, $v=x_{n}$.
\end{definition}

\noindent The following definitions formalize the concept of tracing a sequence by an actual orbit.

\begin{definition}
\cite{KumarDas2024}A dynamical system $(X,f)$ is said to have the 
shadowing property if for every open cover $\mathcal{U}$ of $X$, there exists an open cover $\mathcal{V}$ of $X$ such that any $\mathcal{V}$-pseudo-orbit can be $\mathcal{U}$-traced.

\noindent
The system $(X,f)$ is said to have the finite shadowing property if for every open cover $\mathcal{U}$ of $X$, there exists an open cover $\mathcal{V}$ of $X$ such that any finite $\mathcal{V}$-pseudo-orbit can be $\mathcal{U}$-traced.
\end{definition}

\begin{definition}
     \cite{shah2016}For a uniform space X, $(X,f)$ is said to have topological shadowing property if for any entourage $\mathbf{V}$, there exists an entourage $\mathbf{S}$ such that any $\mathbf{S}$-pseudo orbit can be $\mathbf{V}$-traced.
\end{definition}
\noindent We now turn to various forms of the specification property, formulated in different topological settings.

 \begin{definition}
    \cite{shah2016}Let $(X,f)$ be a uniform space with continuous function $f: X \rightarrow X$ and $\Delta_{X}=\{(x,x) | x \in X \}$. $(X,f)$ has topological specification property if for any symmetric neighborhood $\mathbf{V}$ of $\Delta_{X}$, $\exists M>0$ such that for any $M$-spaced specification $\xi=\left\{f^{\left[a_{i}, b_{i}\right]}\left(x_{i}\right)\right\}_{i=1}^{n},  \exists y \in X$ such that $\left(f^{k}(y), f^{k}\left(x_{i}\right)\right) \in \mathbf{V}, k \in [a_{i} ,b_{i}]; i \in [1,n]$.
\end{definition}
\begin{definition}
    \cite{kumar2024} $(X,f)$ has open cover specification property(OCSP) if for every open cover $\mathcal{U}$, there exists $N \in \mathbb{N}$ such that for any $N$-spaced specification $\xi$ there exists $y \in X$ that $\mathcal{U}$-traces $\xi$ i.e. $(f^k(y),f^k(x_i)) \in U \times U$, for some $U \in \mathcal{U}, k \in [a_i, b_i]; i \in [1,n]$.

    \noindent
    Further if  for some $l\geqslant N+b_n-a_1$, $f^l(y)=y $ then $(X,f)$ is said to have strong open cover specification property, that is, the tracing point can be chosen periodic.
\end{definition}
\begin{definition}
    A dynamical system ($X, f$) is said to have the Quasi weak open cover specification property (QOCSP)  if for any open cover $\mathcal{U}$, there exists a positive $N \in \mathbb{N}$ such that for $x_1, x_2 \in X$ and for any $n \geq N$, there is a point $y \in X$ such that $\left(y, x_1\right) \in U \times U, \text{ for some } U \in \mathcal{U}$ and $\left(f^{n}(y), f^{n}\left(x_2\right)\right)\in V \times V, \text{ for some } V \in \mathcal{U}$.
\end{definition}

\begin{definition}
    \cite{devaney1989}A dynamical system $(X,f)$ is said to be Devaney chaotic if 
    \begin{enumerate}
        \item $(X,f)$ is SDIC,
        \item $(X,f)$ is topologically transitive,
        \item set of periodic points of $X$ is dense in $X$.
    \end{enumerate}
\end{definition}
\noindent Finally, we introduce distributional chaos, a notion of chaotic behavior characterized by the asymptotic distribution of the times at which pairs of orbits remain close or become separated. Several forms of distributional chaos, namely DC1, DC2, and DC3, have been introduced in the literature. In this paper, we focus exclusively on DC1 and use the following definition.

\begin{definition}
\cite{schweizer1994}For $x,y \in X,$ an open cover $\mathcal{U}$ of $X$ and $n \in \mathbb{N}$, define $\zeta(x,y,n, \mathcal{U})=|\{k| k \in [1,n] \text{ and } (f^k(x), f^k(y)) \in U \times U, \text{ for some } U \in \mathcal{U}\}|$ where $|A|$ is the cardinality of the set A. 
By means of $\zeta$, we define $F_{xy}(\mathcal{U})= \liminf_{n \rightarrow \infty} \frac{1}{n} \zeta(x,y,n,\mathcal{U})$ and $F^{*}_{xy}(\mathcal{U})=\limsup_{n \rightarrow \infty}\frac{1}{n} \zeta(x,y,n,\mathcal{U})$.

\noindent 
$(x,y)$ are called pair of distributionally chaotic points if 
\begin{enumerate}
    \item $F_{xy}(\mathcal{U})=0$ for some open cover $\mathcal{U}$ of $X$ and
    \item $F^{*}_{xy}(\mathcal{U})=1$ for every open cover $\mathcal{U}$ of $X$.
\end{enumerate}
Set $\mathbf{S}$ is called a distributionally chaotic scrambled set if every pair of distict points $x,y$ in $\mathbf{S}$ are distributionally chaotic.

\noindent
If a dynamical system $(X,f) $ has an uncountable  distributionally chaotic scrambled set then $(X,f)$ is said to be distributionally chaotic system.
\end{definition}

\begin{remark*}
Many of the above notions coincide with their metric counterparts when $X$ is metrizable.
\end{remark*}

\section{Structural Properties}
We begin by establishing several fundamental structural properties of the open cover specification property. In particular, we show that the property is preserved under taking powers of the dynamical system and is closely related to the corresponding property on product systems. We also establish its invariance under topological conjugacy. These results demonstrate that the open cover specification property is determined by the underlying dynamical structure rather than by a particular representation of the system.

\begin{theorem}
    Let $X$ be an arbitrary, not necessarily compact, topological space and let $f\to X$ be a continuous map. Then the following statements are equivalent to $(X,f)$ having the open cover specification property:
    \begin{enumerate}[label=\roman*.]
    \item If $f$ is a homeomorphism then $(X,f^{p})$ has open cover specification property for every $p \in \mathbb{Z}$
        \item For $X=X_{1} \times X_{2}$ and $f= f_{1} \times f_{2}$ both $(X_{1},f_{1})$ and $(X_{2},f_{2})$ have open cover specification property.
        \item $(Y,g)$ has open cover specification property where $(Y,g)$ is a topological conjugate of $(X,f)$
        \item $(X,f)$ has the topological specification property, given that $X$ is a uniform space.
    \end{enumerate}
\end{theorem}

\begin{proof} $(i)$
The backward implication for $p=1$ is immediate. Hence it is sufficient to show that if $(X,f)$ has OCSP, then $(X,f^{p})$ has OCSP for all $p \in \mathbb{Z}$.

\smallskip
\noindent\emph{\textbf{Case 1: $p \in \mathbb{N}$.}}
Let $\mathcal{U}$ be an open cover of $X$. Since $(X,f)$ has OCSP, there exists $N \in \mathbb{N}$ such that every $N$-spaced specification in $(X,f)$ is $\mathcal{U}$-traced by some point of $X$.

Let $\xi=\bigl\{ (f^{p})^{[a_i,b_i]}(x_i) \bigr\}_{i=1}^n$ be an $N$-spaced specification in $(X,f^{p})$. Then $\xi'=\bigl\{ f^{[pa_i,pb_i]}(x_i) \bigr\}_{i=1}^n$
is an $N$-spaced specification in $(X,f)$. By OCSP of $(X,f)$, there exists $y \in X$ such that for each $i \in \{1,\dots,n\}$ and every $k \in [pa_i,pb_i]$, there exists $U \in \mathcal{U}$ with $(f^{k}(y),f^{k}(x_i)) \in U \times U.$
In particular, for all $\ell \in [a_i,b_i]$,$
\bigl((f^{p})^{\ell}(y),(f^{p})^{\ell}(x_i)\bigr) \in U \times U.
$
Thus $y$ $\mathcal{U}$-traces $\xi$ in $(X,f^{p})$, and $(X,f^{p})$ has OCSP for all $p \in \mathbb{N}$.

\smallskip
\noindent\emph{\textbf{Case 2: $p=-1$.}}
Assume $(X,f)$ has OCSP. Let
$
\xi=\bigl\{ (f^{-1})^{[a_i,b_i]}(x_i) \bigr\}_{i=1}^n
$
be an $N$-spaced specification in $(X,f^{-1})$. Define
$
y_i = f^{-b_n}(x_{n+1-i}), 
c_i = b_n - b_{n+1-i}, 
d_i = b_n - a_{n+1-i},
 i=1,\dots,n.
$
Then
$
\xi'=\bigl\{ f^{[c_i,d_i]}(y_i) \bigr\}_{i=1}^n
$
is an $N$-spaced specification in $(X,f)$. By OCSP of $(X,f)$, there exists $z \in X$ such that for $i \in [1,n]$ and all $k \in [c_i,d_i]$, there exists $U \in \mathcal{U}$ with
$
(f^{k}(z),f^{k}(y_i)) \in U \times U.
$
Equivalently,
$
(f^{-k}(f^{b_n}(z)),f^{-k}(x_i)) \in U \times U, \text{for all } k \in [a_i,b_i].
$
Setting $w = f^{b_n}(z)$, we obtain
$
\bigl((f^{-1})^{k}(w),(f^{-1})^{k}(x_i)\bigr) \in U \times U
\quad \text{for all } k \in [a_i,b_i].
$
Hence, $\xi$ is $\mathcal{U}$-traced in $(X,f^{-1})$, and $(X,f^{-1})$ has OCSP.

\smallskip
Since $(X,f^{-1})$ has OCSP, applying Case~1 to $f^{-1}$ yields that $(X,f^{p})$ has OCSP for all $p \in \mathbb{Z}$.


$(ii)$ 
Assume that $(X_{1} \times X_{2}, f_{1} \times f_{2})$ has OCSP. Let $\mathcal{U}_{1}$ and $\mathcal{U}_{2}$ be arbitrary open covers of $X_{1}$ and $X_{2}$, respectively. It suffices to show that $(X_{1},f_{1})$ has OCSP; the argument for $(X_{2},f_{2})$ is analogous.

Consider the open cover $\mathcal{U}=\{U_{1}\times U_{2} : U_{1}\in\mathcal{U}_{1}, U_{2}\in\mathcal{U}_{2}\}$ of $X_{1}\times X_{2}$. By the OCSP of $(X_{1} \times X_{2}, f_{1} \times f_{2})$, there exists $N\in\mathbb{N}$ such that every $N$-spaced specification in $X_{1}\times X_{2}$ is $\mathcal{U}$-traced by some point of $X_{1}\times X_{2}$.
Let $\xi_{1}=\{f_{1}^{[a_{i},b_{i}]}(x_{i})\}_{i=1}^{n}$ be an $N$-spaced specification in $(X_{1},f_{1})$ and fix an arbitrary point $y\in X_{2}$. Define
$\xi=\{(f_{1}\times f_{2})^{[a_{i},b_{i}]}(x_{i},y)\}_{i=1}^{n}$.
Then $\xi$ is an $N$-spaced specification in $(X_{1}\times X_{2}, f_{1}\times f_{2})$. Hence there exists $(z_{1},z_{2})\in X_{1}\times X_{2}$ such that $\xi$ is $\mathcal{U}$-traced by $(z_{1},z_{2})$. By the definition of $\mathcal{U}$, it follows that $\xi_{1}$ is $\mathcal{U}_{1}$-traced by $z_{1}$. Therefore $(X_{1},f_{1})$ has OCSP.

The same argument, with the roles of $X_{1}$ and $X_{2}$ interchanged, shows that $(X_{2},f_{2})$ has OCSP.



$(iii)$
Let $(X_{1},f_{1})$ and $(X_{2},f_{2})$ be topologically conjugate dynamical systems and homeomorphism $h \colon X_{1}\to X_{2}$ with $h\circ f_{1}=f_{2}\circ h$. Assume that $(X_{1},f_{1})$ has the open cover specification property.

Let $\mathcal{U}_{2}$ be an arbitrary open cover of $X_{2}$,  define an open cover $\mathcal{U}_1$ of $X_{1}$ as $\mathcal{U}_{1}=\{h^{-1}(A): A\in\mathcal{U}_{2}\}$. By the OCSP of $(X_{1},f_{1})$, there exists $N\in\mathbb{N}$ such that for every $N$-spaced specification in $(X_{1},f_{1})$ is $\mathcal{U}_{1}$-traced.
Let $\xi_{2}=\{f_{2}^{[a_{i},b_{i}]}(x_{i})\}_{i=1}^{n}$ be an $N$-spaced specification in $(X_{2},f_{2})$. For each $i$, set $z_{i}=h^{-1}(x_{i})$ and define
$\xi_{1}=\{f_{1}^{[a_{i},b_{i}]}(z_{i})\}_{i=1}^{n}$.
Then $\xi_{1}$ is an $N$-spaced specification in $(X_{1},f_{1})$. Hence there exists $y\in X_{1}$ such that for every $i\in [1,n]$ and every $k\in[a_{i},b_{i}]$, there exists $B\in\mathcal{U}_{1}$ with
$(f_{1}^{k}(y),f_{1}^{k}(z_{i}))\in B\times B$.
Since $B=h^{-1}(A)$ for some $A\in\mathcal{U}_{2}$, this implies
$(f_{1}^{k}(y),f_{1}^{k}(h^{-1}(x_{i})))\in h^{-1}(A)\times h^{-1}(A)$.
Applying $h$ and using the conjugacy relation $h\circ f_{1}=f_{2}\circ h$, we obtain
$(f_{2}^{k}(h(y)),f_{2}^{k}(x_{i}))\in A\times A$.
Thus $\xi_{2}$ is $\mathcal{U}_{2}$-traced by $z=h(y)\in X_{2}$.

Since $\mathcal{U}_{2}$ was arbitrary, it follows that $(X_{2},f_{2})$ has the open cover specification property.




$(iv)$ 
Suppose first that $(X,f)$ has the topological specification property, and let $\mathcal{U}$ be an open cover of $X$.  
Define
$\mathbf{V}=\bigcup_{A\in\mathcal{U}} A\times A.$
Then $\mathbf{V}$ is a symmetric neighborhood of the diagonal $\Delta_X$ in $X\times X$.  
By the topological specification property, there exists $N\in\mathbb{N}$ such that for any points
$x_1,x_2,\ldots,x_n\in X$ and any integers
$a_1\le b_1<a_2\le b_2<\cdots<a_n\le b_n$
satisfying $a_{i+1}-b_i\ge N$ for $i \in [1,n-1]$, there exists $y\in X$ such that
$\bigl(f^k(y),f^k(x_i)\bigr)\in\mathbf{V} \text{ for all } k\in[a_i,b_i];\ i \in [1,n]$.
By the definition of $\mathbf{V}$, this implies that for each $i \in [1,n]$ and each $k\in[a_i,b_i]$ there exists
$A\in\mathcal{U}$ such that
$\bigl(f^k(y),f^k(x_i)\bigr)\in A\times A$.
Hence the specification is $\mathcal{U}$-traced, and $(X,f)$ has the open cover specification property.

Conversely, assume that $(X,f)$ has open cover specification property.  
Let $\mathbf{V}$ be an arbitrary symmetric neighborhood of the diagonal $\Delta_X$.  
For each $x\in X$, there exists an open neighborhood $A_x$ of $x$ such that
$A_x\times A_x\subset\mathbf{V}$.  
The collection
$\mathcal{U}=\{A_x|x\in X, A_x \times A_x \subset \mathbf{V}\}$
is an open cover of $X$.  
By open cover specification property, there exists $N\in\mathbb{N}$ such that every $N$-spaced
specification is $\mathcal{U}$-traced.
Let $\{f^{[a_i,b_i]}(x_i)\}_{i=1}^n$ be an arbitrary $N$-spaced specification in $X$.  
Then there exists $y\in X$ such that for each $i \in [1,n]$ and each $k\in[a_i,b_i]$,
$\bigl(f^k(y),f^k(x_i)\bigr)\in A\times A$ 
for some $A\in\mathcal{U}$.  
Since $A\times A\subset\mathbf{V}$, it follows that
$\bigl(f^k(y),f^k(x_i)\bigr)\in\mathbf{V}
 \text{ for all } k\in[a_i,b_i],\ i\in [1,n]$.
Thus $(X,f)$ has the topological specification property.
\end{proof}

\begin{remark}
\begin{enumerate}[label=\roman*.]
\item The product result in Theorem~2.1(iii) can be extended, by a similar argument, to countable products $(X_1\times X_2\times\cdots,,
f_1\times f_2\times\cdots)$
equipped with the product topology. In particular, if the resulting product system has the open cover specification property, then each factor $(X_i,f_i)$ possesses the open cover specification property.
\item The conjugacy invariance in Theorem~2.1(iv) shows that the open cover specification property is a topological property of the dynamical system. In particular, it is independent of the particular choice of coordinates used to represent the system.

\end{enumerate}
\end{remark}

Following is an example of Non-compact space with open cover specification

\begin{example}
    Let $A=\{0,1\}, \tau=\{\phi,\{1\}, A\}$ be topology on $A$ and 
    let $X \subset A^{\mathbb{Z}}$ denote the set of sequences that contain the symbol 1 at infinitely many coordinates, endowed with the product topology. Thus an open set $\mathcal{U}$ in $X$ will be of the form $\mathcal{U}=\{x \in X | \ x $ contains the symbol 1 at finite or infinitely many coordinates$\}$ since the only non-empty open sets in $A$ are $\{1\}$ and $A$ itself. Consider dynamical system $(X, \sigma)$ where $\sigma$ is shift map, and any open cover $\mathcal{U}$ of $X$. For $N=1$, let $\xi=\left\{\sigma^{\left[a_i, b_i\right]}\left(x_i\right)\right\}_{i=1}^n$ be an $N$-spaced specification.
    For $i \in [1,n]$, let $k$ denote the coordinate corresponding to the $a_i$-th occurrence of the symbol $1$ in $x_i$. Consider $y=\left(y_1, y_2, \ldots\right)$ where $y_k=x_{i,k}$ for $x_i=(x_{i,1},x_{i,2}, \cdots)$. Continuing this assignment up to the coordinate corresponding to the $b_i$-th occurrence of $1$ in $x_i$ we obtain a sequence $y \in X$ that traces the specification $\xi$ with respect to the open cover $\mathcal{U}$ and hence $(X, \sigma)$ has open cover specification property.
\end{example}

The specification property and the shadowing property are two fundamental notions in the study of topological dynamical systems. The shadowing property describes the phenomenon where every approximate trajectory can be closely traced by an actual trajectory. In \cite{kumar2024}, it is proved that on a compact and $T_3$ space with an onto map $f$, shadowing property and topological mixing imply OCSP. We will prove that in compact spaces, the combination of finite shadowing property and chain mixing guarantees OCSP.

\begin{theorem}
Let $(X,f)$ be a compact dynamical system. If $(X,f)$ is chain mixing and has finite shadowing property, then it has open cover specification property.
\end{theorem}

\begin{proof}
Let $\mathcal{U}$ be an arbitrary open cover of $X$.  
Since $(X,f)$ has the finite shadowing property, there exists an open cover $\mathcal{V}$ of $X$ such that every finite $\mathcal{V}$-pseudo orbit is $\mathcal{U}$-traced.
Because $X$ is compact, $\mathcal{V}$ admits a finite subcover
$\mathcal{V}'=\{V_1,V_2,\ldots,V_{N'}\}\subset\mathcal{V}$.
Moreover, since $(X,f)$ is chain mixing implies for any $u,v \in X$ there exists $N''\in \mathbb{N}$ such that there exists $N^{\prime\prime}$ length $\mathcal{V}$-pseudo orbit with initial and terminal points as $u,v$.
Set $N=\max\{N',N''\}$. Then there exists a finite open cover
$\mathcal{W}=\{V_1,V_2,\ldots,V_N\}$
of $X$ with $\mathcal{V}'\subset\mathcal{W}\subset\mathcal{V}$ such that there exists $N$ length $\mathcal{W}$-pseudo orbit with initial and terminal points as $u,v$.

Let $\xi=\{f^{[a_i,b_i]}(x_i)\}_{i=1}^n$
be an arbitrary $N$-spaced specification.  
For each $i \in [2,n]$ with $a_i-b_{i-1}=n_i\ge N$, hence there exist points
$v_{i0},v_{i1},\ldots,v_{in_i}\in X$
such that $v_{i0}=f^{b_{i-1}}(x_{i-1}), v_{in_i}=f^{a_i}(x_i)$,
and $\bigl(f(v_{i(j-1)}),v_{ij}\bigr)\in V\times V
\quad \text{for some } V\in\mathcal{W}, \ j \in [1,n_i]$.
Consider the finite sequence obtained by concatenating the orbit segments
\[
x_1,f(x_1),\ldots,f^{b_1}(x_1),
\]
the connecting chains $\{v_{ij}\}$, and the orbit segments
\[
f^{a_i}(x_i),\ldots,f^{b_i}(x_i), \quad i=2,\ldots,n.
\]
This sequence is a finite $\mathcal{V}$-pseudo-orbit.  
By the finite shadowing property, there exists $y\in X$ that $\mathcal{U}$-traces this pseudo-orbit.

Consequently, for each $i \in [1,n]$ and each $k\in[a_i,b_i]$, there exists $U\in\mathcal{U}$ such that $\bigl(f^k(y),f^k(x_i)\bigr)\in U\times U$.
Thus every $N$-spaced specification is $\mathcal{U}$-traced, and $(X,f)$ has open cover specification property.
\end{proof}

\section{Open Cover Specification and Chaos}
The study of chaos provides a natural setting for investigating the dynamical consequences of specification properties. In particular, specification-type properties have been shown to impose strong forms of orbit complexity and to be closely related to several notions of chaos. In this section, we investigate the relationship between the open cover specification property and two important notions of chaotic behavior, namely distributional chaos and Devaney chaos.

The relationship between specification properties and distributional chaos has been studied extensively. Oprocha and \v{S}tef'ankov'a~\cite{oprocha2008} established a connection between the specification property and distributional chaos, while Yadav and Shah~\cite{yadav2022} investigated related connections between topological weak specification and distributional chaos in noncompact spaces. Motivated by these results, we first study the implications of the open cover specification property for distributional chaos. In particular, we show that, on a locally compact space, the existence of a pair of distal points together with the open cover specification property leads to distributional chaos of type DC1.

\begin{lemma}\label{lem:Lemma 4.1}
    Let $X$ be a locally compact Hausdorff
    space, 
    $x\in X$, and $\mathcal{U}$ an open cover of $X$. Then $\mathcal{U}$ has an open refinement that is locally finite at the point $x$.
\begin{proof}
    Since $X$ is locally compact Hausdorff, there exists an open neighborhood $V$ of $x$ such that $x\in V \subset \overline{V} \text{ and } \overline{V}\ \text{is compact}.$
$\mathcal{U}$ is an open cover of $X$ and $\overline{V}$ is compact set implies there exist finitely many sets $U_1, U_2, \cdots, U_n \in \mathcal{U}$
such that $\overline{V} \subset \bigcup_{i=1}^n U_i.$
Define $\mathcal{W}
=
\{U_1,\dots,U_n\}
\;\cup\;
\{U\in\mathcal U : U\cap \overline{V}=\varnothing\}.$
$\mathcal{W}$ is an open refinement of $\mathcal U$ and  is locally finite at $x$:  
The neighborhood $V$ of $x$ intersects only the finitely many sets $U_1,\dots,U_n$.  
All other sets in $\mathcal{W}$ are disjoint from $\overline{V}$, hence from $V$.
Thus, only finitely many members of \(\mathcal V\) meet a neighborhood of \(x\).
\end{proof}
\end{lemma}

\begin{theorem}
	Let $(X,f)$ be a dynamical system on a first countable, locally compact Hausdorff space $X$. If $(X,f)$ satisfies the open cover specification property for $s=2$ and admits a pair of distal points then $(X,f)$ is distributionally chaotic.
\end{theorem}

\begin{proof}
	Let $u,v$ be a distal pair. By assumption, there exists an open cover $\mathcal U$ of $X$ such that for every $i\in\mathbb N\cup\{0\}$,
	\(
	(f^{i}(u),f^{i}(v)) \notin U\times U \quad \text{for any } U\in\mathcal U.
	\)
	
	Since $(X,f)$ has open cover specification property, choose a subsequence
	$\{n_i\}_{i\in\mathbb N}$ of positive integers such that the transition times required
	by the specification property are negligible compared to the gaps
	$n_{i+1}-n_i$ (The terms of sequence $n_i$ will be specified as per the open covers defined later). Moreover, assume that
	$
	\lim_{i\to\infty}\frac{n_i-n_{i-1}}{n_{i-1}}=\infty.
	$
	The idea is to alternate between very long intervals where orbit segments
	remain close (lying in the same element of $\mathcal U$) and very long
	intervals where they are far apart, with each successive block dominating
	the cumulative effect of the previous ones.
	
	Define
	$
	U_u=\bigcap\{U\in\mathcal U : u\in U\},
	U_v=\bigcap\{U\in\mathcal U : v\in U\}.
	$
    Since $X$ is locally compact Hausdorff space, by Lemma \ref{lem:Lemma 4.1}, $U_u$ and $U_v$ are open sets.
	$X$ is first countable, hence there exist decreasing neighbourhood bases
	$\{U_n\}$ at $u$ and $\{V_n\}$ at $v$ (By decreasing we mean $U_n\supset U_{n+1}$). Without loss of generality,  we can assume
	$U_1\subset U_u$ and $V_1\subset U_v$. For each $n$, choose open sets
	$U_n',V_n'$ such that
	$
	U_n' \subset \overline{U_n'} \subset U_n,
	V_n' \subset \overline{V_n'} \subset V_n.
	$
	
	Let $\{\mathcal U_n\}_{n\in\mathbb N}$ be the sequence of open covers defined by
	$
	\mathcal U_n=\{U_n, V_n, X\setminus(\overline{U_n'}\cup \overline{V_n'})\}.
	$
	
	\medskip
	\noindent\emph{\textbf{Step 1.}}
	Using the open cover specification property for $\mathcal U_2$, choose
	integers $n_0<n_1<n_2$ and points $y_u,y_v$ such that $y_u$
	$\mathcal U_2$-traces the orbit segments $(u,u)$ on $[n_0,n_1]$ and
	$[n_1,n_2]$, while $y_v$ $\mathcal U_2$-traces $(u,v)$ on these intervals.
	Define
	$	A_u=\bigcap_{k=n_0}^{n_2} f^{-k}(U_k^1),
	A_v=\bigcap_{k=n_0}^{n_2} f^{-k}(U_k^2),
	$
	where $U_k^1,U_k^2\in\mathcal U_2$ are the corresponding open sets in the open cover $\mathcal U_2$.
	Since $X$ is locally compact Hausdorff, there exist nonempty compact
	subsets $B_u\subset A_u$ and $B_v\subset A_v$.
	
	\medskip
	\noindent\emph{\textbf{Step 2.}}
	Let $n_6>n_5>\cdots>n_2$ and consider points
	$y_{uu},y_{uv},y_{vu},y_{vv}$ which $\mathcal U_6$-trace the orbit patterns
	\[
	(u,u,u,v,u,u),\quad (u,u,u,v,u,v),\quad
	(u,v,u,v,v,u),\quad (u,v,u,v,v,v)
	\]
	respectively on the intervals $[n_0,n_1],\dots,[n_5,n_6]$.
	Define
	\[
	\begin{aligned}
		A_{uu}&=\bigcap_{k=n_0}^{n_5} f^{-k}(U_k^1), \qquad
		A_{uv}=\bigcap_{k=n_0}^{n_5} f^{-k}(U_k^2),\\
		A_{vu}&=\bigcap_{k=n_0}^{n_5} f^{-k}(U_k^3), \qquad
		A_{vv}=\bigcap_{k=n_0}^{n_5} f^{-k}(U_k^4).
	\end{aligned}
	\]
	For each $i,j\in\{u,v\}$, choose nonempty compact subsets
	$B_{ij}\subset A_{ij}$. Proceeding inductively, for any $m\in\mathbb N$
	we obtain nonempty compact sets $B_{\alpha_1\alpha_2\cdots\alpha_m}$
	indexed by finite sequences $\alpha_i\in\{u,v\}$, such that for any $\alpha_i\in \{u,v\}$ each point in
	$B_{\alpha_1\cdots\alpha_m}$,  $\mathcal U_n$ traces(for an appropriate $n>m$) the orbit pattern $u,u_{\alpha_1},u,v,u_{\alpha_1},u_{\alpha_2},u,v,u_{\alpha_1},u_{\alpha_2},u_{\alpha_3},u,v\dots,u,v,u_{\alpha_1},u_{\alpha_2},\dots,u_{\alpha_m}$
	
	By the Cantor intersection theorem, for every infinite sequence
	$\alpha=(\alpha_1,\alpha_2,\ldots)\in\{u,v\}^{\mathbb N}$,
	$
	\mathcal B_\alpha=\bigcap_{n=1}^\infty B_{\alpha_1\cdots\alpha_n}
	\neq\varnothing.
	$
	Fix one point from each $\mathcal B_\alpha$ and define
	$
	S=\{b\in\mathcal B_\alpha : \alpha\in\{u,v\}^{\mathbb N}\}.
	$
	Since $\{u,v\}^{\mathbb N}$ is uncountable, so is $S$.
	
	Let $x,y\in S$ be distinct. Then $x\in\mathcal B_\alpha$ and
	$y\in\mathcal B_\beta$ for sequences $\alpha\neq\beta$, by construction of the sequences $\alpha,\beta$, we know that the two  differ at
	infinitely many indices. Without loss of generality, we can take $\alpha_i=u$ and $\beta_i=v$. The
	corresponding orbit segments of $x$ and $y$ lie in disjoint elements of
	$\mathcal U$, implying
	$
	F_{xy}(\mathcal U)=0.
	$
	On the other hand, for any open cover $\mathcal V$ of $X$, there exists an $n_0\in \mathbb N$ such that $U_{n_0}\subset \bigcap\{V\in \mathcal V: u\in V\}$.  Then as
	construction ensures that $x$ and $y$,  $\mathcal U_n$ traces $u$ simultaneously on
	infinitely many long intervals, yielding
	$
	F_{xy}^*(\mathcal V)=1.
	$
	
	Thus, $S$ is an uncountable distributionally scrambled set for $f$, and
	consequently $(X,f)$ is distributionally chaotic of type~$1$.
\end{proof}
The above result immediately yields the following consequence, since distributional chaos of type DC1 implies Li--Yorke chaos.
\begin{definition}
Let $\mathcal U$ be open cover of $X$. Define
$$C_1(f)=\left\{(x,y)\in X^2 \,\middle|\, \exists\,\langle n_j\rangle\subset\mathbb N,(f^{n_j}(x),f^{n_j}(y))\in U\times U \text{ for some } U\in\mathcal U \right\}$$
$$C_2(f)=\left\{(x,y)\in X^2 \,\middle|\, \exists\,\langle n_j\rangle\subset\mathbb N,(f^{n_j}(x),f^{n_j}(y))\notin U\times U \text{ for all } U\in\mathcal U \right\}.$$
The dynamical system $(X,f)$ is said to be \emph{Li--Yorke chaotic} if there exists an open cover $\mathcal U$ of $X$ and an uncountable set $S\subset X$ (called a \emph{scrambled set}) such that
$S\times S \subseteq C_1(f)\cap C_2(f).$
\end{definition}

 \begin{corollary}
Let $(X,f)$ be a dynamical system on a locally compact space $X$. If $(X,f)$ has the open cover specification property for $s=2$ and admits a pair of distal points, then $(X,f)$ is Li--Yorke chaotic.
\end{corollary}

\begin{proof}
If $(X,f)$ has open cover specification property, then it is distributionally chaotic. Consequently, there exists an uncountable scrambled set $S\subset X$ with respect to some open cover $\mathcal U$ such that, for all $x,y\in S$,
$\liminf_{n\to\infty}\frac{1}{n}\bigl|\{k\in[1,n] : (f^k(x),f^k(y))\in U\times U \text{ for some } U\in\mathcal U\}\bigr|=0
$
and
$\limsup_{n\to\infty}\frac{1}{n}\bigl|\{k\in[1,n] : (f^k(x),f^k(y))\in U\times U \text{ for some } U\in\mathcal U\}\bigr|=1.$
It follows that for any $x,y\in S$ there exist subsequences $\{N_j\}$ and $\{M_j\}$ of $\mathbb N$ such that
$(f^{N_j}(x),f^{N_j}(y))\notin U\times U \quad \text{for all } \in\mathcal U,$
and $ f^{M_j}(x),f^{M_j}(y) \in U\times U \text{ for some } U\in\mathcal U.$
Hence $S\times S\subseteq C_2(f)$ and $S\times S\subseteq C_1(f)$, and therefore $S\times S\subseteq C_1(f)\cap C_2(f).$
Thus $(X,f)$ is Li--Yorke chaotic.
\end{proof}

We next investigate the relationship between the open cover specification property and Devaney chaos. Recall that a dynamical system $(X,f)$ is said to be Devaney chaotic if it is topologically transitive and has a dense set of periodic points. On a suitable topological space, these two conditions imply sensitivity to initial conditions. Thus, Devaney chaos provides a natural framework for studying the chaotic behavior generated by specification properties.

The following result shows that the strong open cover specification property is sufficient to guarantee Devaney chaos under a mild topological assumption.

\begin{lemma}\label{lem:Lemma 4.4}
If $(X,f)$ has the QOCSP, then $(X,f)$ is topologically mixing.
\end{lemma}

\begin{proof}
Let $A,B$ be nonempty open subsets of $X$. Choose points $x\in A$ and $y\in B$. Let $\mathcal U$ be an open cover of $X$ such that
\begin{enumerate}[itemsep=0pt, topsep=0pt]
    \item $A,B\in\mathcal U$;
    \item if $x\in U\in\mathcal U$, then $U\subseteq A$;
    \item if $y\in V\in\mathcal U$, then $V\subseteq B$.
\end{enumerate}
Since $(X,f)$ has the QOCSP, there exists $M\in\mathbb N$ such that for any $z_1,z_2\in X$ and every $n\ge M$, there exists $z\in X$ and $U,V\in\mathcal U$ with $(z,z_1)\in U\times U \text{and} (f^n(z),f^n(z_2))\in V\times V$.
Applying this with $z_1=x$ and $z_2=f^{-n}(y)$, we obtain $z\in X$ such that $(z,x)\in U\times U$ for some $U\in\mathcal U$ and $(f^n(z),y)\in V\times V$ for some $V\in\mathcal U$. By construction of $\mathcal U$, this implies $z\in A$ and $f^n(z)\in B$. Hence $A\cap f^{-n}(B)\neq\varnothing \text{for all } n\ge M,$
and therefore $(X,f)$ is topologically mixing.
\end{proof}

\begin{lemma}\label{lem:4.5}
For a $T_3$ space $X$, if $(X,f)$ has the QOCSP, then $(X,f)$ has sensitive dependence on initial conditions.
\end{lemma}

\begin{proof}
Let $y_1,y_2\in X$ be distinct points. Since $X$ is $T_3$, there exist disjoint open sets $U,V\subseteq X$ with $y_1\in U$ and $y_2\in V$, and open sets $Y_1,Y_2$ such that $y_1\in Y_1, \overline{Y_1}\subset U,\
y_2\in Y_2, \overline{Y_2}\subset V$.
Define the open cover $\mathcal U=\{U,V,W\}, \text{ where } W=X\setminus\overline{Y_1\cup Y_2}.$
We show that $\mathcal U$ is a sensitivity cover.

Let $x\in X$ and let $G$ be any open neighborhood of $x$. Without loss of generality, assume $G\cap(U\cup V)=\varnothing$. Define the open cover $\mathcal U_1=\{Y_1,Y_2,G,X\setminus\{x,y_1,y_2\}\}$.
Since $(X,f)$ has the QOCSP, there exists $M\in\mathbb N$ such that for all $n\ge M$ and any $z_1,z_2\in X$, there exists $w\in X$ satisfying $(w,z_1)\in U'\times U' \quad\text{and}\quad (f^n(w),f^n(z_2))\in V'\times V'$
for some $U',V'\in\mathcal U_1$.
Fix $n\ge M$. Taking $z_1=x$ and $z_2=f^{-n}(y_1)$, there exists $w_1\in X$ such that $(x,w_1)\in G\times G$ and $(f^n(w_1),y_1)\in Y_1\times Y_1$. Similarly, taking $z_2=f^{-n}(y_2)$, there exists $w_2\in X$ such that $(x,w_2)\in G\times G$ and $(f^n(w_2),y_2)\in Y_2\times Y_2$.
Consequently, at least one of the pairs $(f^n(x),f^n(w_1))$ or $(f^n(x),f^n(w_2))$ does not lie in $A\times A$ for any $A\in\mathcal U$. This establishes sensitive dependence on initial conditions.
\end{proof}

\begin{lemma}\label{lem:4.6}
Let $X$ be a $T_1$ space. If $(X,f)$ has strong OCSP, then the set of periodic points of $f$ is dense in $X$.
\end{lemma}

\begin{proof}
Let $A\subseteq X$ be a nonempty open set and let $x\in A$. Consider open cover $\mathcal U=\{A,X\setminus\{x\}\}$. By the strong OCSP, there exists $N\in\mathbb N$ such that any $N$-spaced specification $\xi$ is $\mathcal U$-traced by some point in $X$.

Consider the specification $\xi=\{f^{[0,0]}(x)\}$. Then there exists $y\in X$ such that $(x,y)\in A\times A$ and $f^N(y)=y$. Hence $y$ is a periodic point contained in $A$. Since $A$ was arbitrary, the set of periodic points is dense in $X$.
\end{proof}

\begin{theorem}
Let $(X,f)$ be a dynamical system with the strong OCSP, where $X$ is a $T_3$ space. Then $(X,f)$ is Devaney chaotic.
\end{theorem}

\begin{proof}
Since strong OCSP implies QOCSP, the system $(X,f)$ is topologically mixing by Lemma \ref{lem:Lemma 4.4} and has sensitive dependence on initial conditions by Lemma \ref{lem:4.5}. By Lemma \ref{lem:4.6}, the set of periodic points is dense in $X$. Hence $(X,f)$ satisfies all three conditions of Devaney chaos.
\end{proof}

  The $T_3$ assumption in the above theorem is sufficient but not necessary. The following example demonstrates that a dynamical system may possess the open cover specification property and be Devaney chaotic even when the underlying space is not $T_3$.

\begin{example}
Let $X=[0,1]$ and equip $X$ with the topology \\ $\tau= \left\{ \phi, \left\{ \frac{1}{2} \right\} \cup U
\middle| U\text{ is open in the usual topology on }[0,1] \right\} .$
Consider the dynamical system $(X,f)$, where $f$ is the tent map. Then $(X,f)$ has the open cover specification property and is Devaney chaotic, as in the case of the usual topology. However, $(X,\tau)$ is not a $T_3$ space. Hence, the $T_3$ assumption in the preceding theorem is not necessary.
\end{example}

\noindent
\textbf{Acknowledgments}:  
The first author gratefully acknowledges financial support from the University Grants Commission, India, under the grant code BMHMU00886752 U.

\end{document}